\documentclass[12pt]{amsart}
\usepackage[all]{xy}

\usepackage{graphicx}  

\usepackage{amsmath}
\usepackage{amssymb}
\usepackage{amsfonts}
\usepackage{mathrsfs}
\usepackage{mathtools}

\newcommand{\Cdb}{\ensuremath{\mathbb{C}}}

\newcommand{\A}{\mbox{${\mathcal A}$}}
\newcommand{\B}{\mbox{${\mathcal B}$}}

\newcommand{\D}{\mbox{${\mathcal D}$}}

\renewcommand{\S}{\mbox{${\mathcal S}$}}

\newcommand{\norm}[1]{\Vert#1\Vert}

\newcommand{\cbnorm}[1]{\Vert#1\Vert_{cb}}
\newcommand{\decnorm}[1]{\Vert#1\Vert_{dec}}

\newcommand{\ma}{\otimes_{\rm max}}

\newtheorem{theorem}{Theorem}[section]
\newtheorem{lemma}[theorem]{Lemma}
\newtheorem{corollary}[theorem]{Corollary}
\newtheorem{proposition}[theorem]{Proposition}

\theoremstyle{remark}
\newtheorem{remark}[theorem]{\bf Remark}
\theoremstyle{definition}

\numberwithin{equation}{section}

\begin{document}

\title[]{Decomposability of bimodule maps}

\author{Christian Le Merdy and Lina Oliveira}
\address{Laboratoire de Math\'ematiques\\ Universit\'e de  Franche-Comt\'e
\\ 25030 Besan\c con Cedex\\ France}
\email{clemerdy@univ-fcomte.fr}
\address{Center for Mathematical Analysis,
Geometry and Dynamical Systems
{\sl and}
Department  of Mathematics\\Instituto Superior 
T\'ecnico\\
Universidade de Lisboa\\
Av. Rovisco Pais\\
1049-001 Lisboa\\
Portugal}
\email{linaoliv@math.ist.utl.pt}

\date{\today}

\thanks{The first author is supported by the research program ANR 2011 BS01 008 01. The second author is partially supported by CAMGSD-LARSyS through the FCT Program POCTI-FEDER and 
 by the FCT project EXCL/MAT-GEO/0222/2012.}

\begin{abstract} Consider a unital $C^*$-algebra $A$, a von Neumann algebra $M$,
a unital sub-$C^*$-algebra $C\subset A$ and a unital $*$-homomorphism $\pi\colon C\to M$.
Let $u\colon A\to M$ be a decomposable map (i.e. a linear combination of completely positive maps)
which is a $C$-bimodule map with respect to $\pi$. We show that $u$ is 
a linear combination of $C$-bimodule completely positive maps if and only if there exists a 
projection $e\in \pi(C)'$ such that $u$ is valued in $eMe$ and $e\pi(\cdotp)e$ 
has a completely positive extension
$A\to eMe$.
\end{abstract}

\maketitle

\bigskip\noindent
{\it 2000 Mathematics Subject Classification : 46L07, 46L05.}

\bigskip
\section{Introduction and preliminaries}

Let $A$ be a unital $C^*$-algebra, let $C\subset A$ be a 
sub-$C^*$-algebra containing the unit,
let $M$ be another unital $C^*$-algebra and let
$\pi\colon C\to M$ be a unital $*$-homomorphism.
We say that a bounded linear map $u\colon A\to M$ 
is a $C$-bimodule map with respect to $\pi$ if
\begin{equation}\label{Bim}
u(c_1ac_2)=\pi(c_1)u(a)\pi(c_2)
\end{equation}
for any $a\in A$ and any $c_1,c_2\in C$. In the sequel we will
simply say a `$C$-bimodule map' if the homomorphism 
$\pi$ to which it refers is clear.

It was shown in \cite{W} that if $M$ is injective, 
then any $C$-bimodule completely bounded map $u\colon A\to M$
can be decomposed as a linear combination of four completely positive
$C$-bimodule maps from $A$ into $M$. 
The aim of this note is to extend that result to the more general
context of decomposable operators. Note that in the sequel, $M$ is no 
longer assumed to be injective.

\bigskip
We will assume that the reader is familiar with the notions of 
completely positive maps and completely bounded 
maps, see \cite{Pa} for an introduction.
A linear map $u\colon A\to M$ is called decomposable if there exist four completely positive 
maps $u_1,u_2,u_3,u_4\colon A\to M$ such that $u=(u_1-u_2)+i(u_3-u_4)$. Let $u_*\colon A\to M$ 
be defined by $u_*(a)=u(a^*)^*$. According to
\cite{H}, $u$ is decomposable if and only if there exist two
completely positive maps $S_1,S_2\colon A \to M$ such that the mapping
$$
\begin{pmatrix} a_1 & t \\ s & a_2 \end{pmatrix}\,\mapsto\,
\begin{pmatrix} S_1(a_1) & u(t) \\ u_*(s) & S_2(a_2)\end{pmatrix}
$$
from $M_2(A)$ into $M_2(M)$ is completely positive. Furthermore letting
$\decnorm{u}=\inf\bigl\{\norm{S_1},\norm{S_2}\bigr\}$ with infimum
taken over all possible pairs $(S_1,S_2)$ satisfying the above property 
defines a norm on the space of decomposable operators. Moreover
$\cbnorm{u}\leq\decnorm{u}$ for any
decomposable operator $u\colon A\to M$ and $\cbnorm{u} =\decnorm{u}$
when $M$ is injective.

A natural question is whether a $C$-bimodule decomposable map
from $A$ into $M$ can be decomposed as a linear combination of 
completely positive $C$-bimodule maps $A\to M$. We do not know if this 
always holds true. In the case when
$M$ is a von Neumann algebra, we will show that 
this holds when $\pi$ extends to a completely positive map
from $A$ into $M$. This extension property is automatically satisfied
when $M$ is injective (so that our result is formally an extension
of Wittstock's Theorem) and also when $C$ has the weak expectation property
(see Corollary \ref{Wep}). Further we will show in Proposition \ref{e1} 
that this extension property is somehow unavoidable.

We end this section with a $C$-bimodule version of the $2\times 2$ matrix characterization
of decomposability reviewed above. Its proof is similar to the one in the classical case so we omit
it. We regard the direct sum $C\oplus C$ as a unital sub-$C^*$-algebra of $M_2(A)$ by identifying
that algebra with the set of diagonal matrices with entries in $C$. Then we let $\pi\oplus\pi\colon 
C\oplus C\to M_2(M)$ be the $*$-homomorphism sending 
$\begin{pmatrix} c_1 & 0 \\ 0 & c_2 \end{pmatrix}$ to 
$\begin{pmatrix} \pi(c_1) & 0 \\ 0 & \pi(c_2) \end{pmatrix}$ for any $c_1,c_2$ in $C$.

\begin{lemma}\label{dec}
Let $u\colon A\to M$ be a bounded linear map, the following assertions are equivalent.
\begin{itemize}
\item [(i)] There exist $C$-bimodule completely positive maps
$u_1,u_2,u_3,u_4\colon A\to M$ such that $u=(u_1-u_2) + i(u_3-u_4)$.
\item [(ii)] There exist two $C$-bimodule completely positive
maps $S_1,S_2\colon A \to M$ such that 
$$
\begin{pmatrix} a_1 & t \\ s & a_2 \end{pmatrix}\,\mapsto\,
\begin{pmatrix} S_1(a_1) & u(t) \\ u_*(s) & S_2(a_2)\end{pmatrix}
$$
is a completely positive map from $M_2(A)$ into $M_2(M)$.
\item [(iii)] There exist two bounded
maps $S_1,S_2\colon A \to M$ such that the map
$$
\begin{pmatrix} a_1 & t \\ s & a_2 \end{pmatrix}\,\mapsto\,
\begin{pmatrix} S_1(a_1) & u(t) \\ u_*(s) & S_2(a_2)\end{pmatrix}
$$
from $M_2(A)$ into $M_2(M)$ is completely positive and a $(C\oplus C)$-bimodule map
with respect to $\pi\oplus\pi$.
\end{itemize}
\end{lemma}

\section{Decomposition into $C$-bimodule completely positive maps}

We will use the following classical lemma (see e.g. \cite[1.3.12]{BLM} or \cite[Lem. 14.3]{P}).

\begin{lemma}\label{lem}
Let $\A,\B$ be $C^*$-algebras, let $\D\subset\B$
be a sub-C$^*$-algebra and let 
$V\colon \B\to \A$ be a contractive completely positive map
such that the restriction $V_{\vert \footnotesize{\D}}$ is a $*$-homomorphism.
Then 
$$
V(bd)=V(b)V(d)\quad\hbox{and} \quad V(db)=V(d)V(b),\qquad b\in \B,\, d \in\D.
$$
\end{lemma}

Our main result is the following theorem. Its proof combines and extends techniques from
\cite{S} and \cite[Chapter 14]{P}. We will use the maximal tensor product of $C^*$-algebras, see e.g. 
\cite[Chapter 12]{Pa} or \cite[Chapter 11]{P}
for some background.

\begin{theorem}\label{main} Let $M$ be a von Neumann algebra, let $A$ be a unital
$C^*$-algebra, let $C\subset A$ be a sub-$C^*$-algebra containing the unit 
and let $\pi\colon C\to M$ be a unital $*$-homomorphism. 
Let $u\colon A\to M$ be a $C$-bimodule map (with respect to $\pi$) and assume that 
$u$ is decomposable. If $\pi$ admits a completely positive extension
$A\to M$ then there exist $C$-bimodule completely positive maps
$u_1,u_2,u_3,u_4\colon A\to M$ such that 
$u=(u_1-u_2) + i(u_3 -u_4)$.
\end{theorem}

\begin{proof}
We may assume that $\decnorm{u}=1$. 
Consider $M\subset B(H)$ for some Hilbert space $H$  and let
$M'\subset B(H)$ be the corresponding commutant. Let $w\colon A\otimes M'\to B(H)$
be the linear map taking $a\otimes y$ to $u(a)y$ for any $a\in A$ and any $y\in M'$.
Since $u$ is decomposable, $w$ extends to a completely bounded map
$$
w\colon A\ma M'\longrightarrow B(H),
$$
with $\cbnorm{w}\leq 1$; see \cite[Eq. (11.6) and Thm. 14.1]{P}.

Let $\widehat{\pi}\colon A\to M$ be a completely positive extension of $\pi$ (as given 
by the assumption). Similarly there exists a completely bounded map 
$\sigma\colon A\ma M'\to B(H)$ such that $\sigma(a\otimes y) = \widehat{\pi}(a)y$
for any $a\in A$ and any $y\in M'$.

We let $\alpha$ be the $C^*$-norm on $C\otimes M'$ induced by $A\ma M'$, 
so that 
$$
C\otimes_\alpha M'\subset A\ma  M'. 
$$
For simplicity
we let $B=A\ma M'$ and $D=C\otimes_\alpha M'$.
Next we let 
$$
\rho\colon D \longrightarrow B(H)
$$
be the restriction
of $\sigma$ to $D$. It is clearly a unital $*$-representation.

For any $a\in A$, $c\in C$ and $y,z\in M'$, we have 
\begin{align*}
w \bigl((c\otimes z)(a\otimes y)\bigr)  = w(ca\otimes zy) 
& = u(ca)zy\\
& = \pi(c)u(a)zy\\
& = \pi(c)zu(a)y\\
& = \rho(c\otimes z)w(a\otimes y)
\end{align*}
because $u$ is a $C$-bimodule map and $z$ commutes with $u(a)\in M$.
Likewise we have
$$
w\bigl((a\otimes y)(c\otimes z)\bigr) = w(a\otimes y)\rho(c\otimes z).
$$
By linearity and continuity, this implies that $w$ is a $D$-bimodule map with respect to
$\rho$.

We consider the `bimodule Paulsen system' associated to $D\subset B$, i.e.
$$
\S\,=\biggl\{
\begin{pmatrix} d_1 & x \\ y & d_2 \end{pmatrix}\, :
\, x,y\in B,\ d_1,d_2\in D\biggr\}\,\subset\, M_2(B).
$$
It is well-known that the bimodule property of $w$ and the norm condition 
$\cbnorm{w}\leq 1$ imply that the map $W\colon\S\to M_2(B(H))$ defined by
$$
W\colon \begin{pmatrix} d_1 & x \\ y & d_2\end{pmatrix}
\,\mapsto\, \begin{pmatrix} \rho(d_1) & w(x) \\ w_*(y) & \rho(d_2)\end{pmatrix}
$$
is a unital completely positive map (see e.g. \cite[3.6.1]{BLM}).
By Arveson's extension Theorem (see e.g. \cite[Thm. 7.5]{Pa}), $W$ 
admits a completely positive extension
$$
\widehat{W}\colon M_2(B)\longrightarrow M_2(B(H)). 
$$
Regard $D\oplus D\subset \S$
as the subspace of $2\times 2$ diagonal matrices with entries in $D$. By 
construction, the restriction of $\widehat{W}$ to that unital $C^*$-algebra is the unital
$*$-representation
$$
\rho\oplus\rho\colon D\oplus D\longrightarrow M_2(B(H)).
$$
By Lemma \ref{lem}, this implies that 
$\widehat{W}$ is a $D\oplus D$-bimodule map with respect to
$\rho\oplus\rho$.

In particular $\widehat{W}$ is a $\Cdb\oplus \Cdb$-bimodule map, which implies 
that it is `corner preserving' in the sense of \cite[2.6.15]{BLM}. More precisely,
there exist two completely positive maps $\Gamma_1,\Gamma_2\colon B\to B(H)$ such that
$$
\widehat{W}\begin{pmatrix} b_1 & x \\ y & b_2\end{pmatrix} = 
\begin{pmatrix} \Gamma_1(b_1) & w(x) \\ w_*(y) & \Gamma_2(b_2)\end{pmatrix}
$$
for any $x,y,b_1,b_2$ in $B$ (see \cite[2.6.17]{BLM}). Now the fact that
$\widehat{W}$ is a $D\oplus D$-bimodule map ensures that
$\Gamma_1$ and $\Gamma_2$ are both $D$-bimodule maps (with respect to $\rho$).

For $j=1,2$, we define $S_j\colon A\to B(H)$ by letting 
$S_j(a)=\Gamma_j(a\otimes 1)$ for any $a\in A$. Since 
$w(a\otimes 1) =u(a)$, the linear map
$\Phi\colon M_2(A)\to M_2(B(H))$ given by 
$$
\Phi\colon \begin{pmatrix} a_1 & t \\ s & a_2\end{pmatrix}
\,\mapsto\, \begin{pmatrix} S_1(a_1) & u(t) \\ u_*(s) & S_2(a_2)\end{pmatrix}
$$
is the restriction of $\widehat{W}$ to $A\simeq A\otimes 1$. Thus 
$\Phi$ is completely positive.

Let $a\in A$. For any $z\in M'$ we have
$$
(a\otimes 1)(1\otimes z) = a\otimes z = (1\otimes z)(a\otimes 1).
$$
Since $\Gamma_1$ is a $D$-bimodule map, this implies that 
$$
\Gamma_1(a\otimes 1)\rho(1\otimes z) = \rho(1\otimes z)\Gamma_1(a\otimes 1).
$$
Equivalently, $S_1(a)z=zS_1(a)$. This shows that $S_1(a)\in M$. 
The same holds for $S_2$ and hence $\Phi$ is valued in $M_2(M)$.

According to Lemma \ref{dec}, it therefore suffices to show that $S_1,S_2$ are
$C$-bimodule maps. For that purpose consider $a\in A, c_1,c_2\in C$.
Since $\Gamma_1, \Gamma_2$ are $D$-bimodule maps, we have
\begin{align*}
S_j(c_1ac_2) = \Gamma_j(c_1ac_2\otimes 1)
& = \Gamma_j\bigl((c_1\otimes 1)(a\otimes 1)(c_2\otimes 1)\bigr)\\
& = \rho(c_1\otimes 1)\Gamma_j(a\otimes 1)\rho(c_2\otimes 1)\\
& = \pi(c_1) S_j(a)\pi(c_2),
\end{align*}
which proves the result. 
\end{proof}

If $M$ is injective, then $\pi$ has a completely positive extension
$A\to M$. Hence Theorem \ref{main} reduces to Wittstock's Theorem in this
case.

Let $\kappa_C\colon C\hookrightarrow C^{**}$ be the canonical embedding.
By definition a $C^*$-algebra $C$ has the weak expectation property (WEP in short) provided that
for one (equivalently for any) $C^*$-algebra embedding $C\subset B(K)$, there exists a
contractive and completely positive map $P\colon B(K)\to C^{**}$ whose rectriction 
to $C$ equals $\kappa_C$. This notion goes back to Lance \cite{La}. The class of $C^*$-algebras with 
the WEP includes nuclear ones and injective ones. 

\begin{corollary}\label{Wep} Let $A,C,M,\pi$ as in Theorem \ref{main}.
Assume that $C$ has the WEP. Let $u\colon A\to M$ be a $C$-bimodule decomposable 
map. Then there exist $C$-bimodule completely positive maps
$u_1,u_2,u_3,u_4\colon A\to M$ such that 
$u=(u_1-u_2) + i(u_3 -u_4)$.
\end{corollary}

\begin{proof}
By Theorem \ref{main} it suffices to show that $\pi$ has a completely
positive extension $\widehat{\pi}\colon A\to M$. 
Let $J\colon A\hookrightarrow B(K)$ be a $C^*$-algebra embedding for a suitable
$K$. By assumption there exists a completely positive $P\colon B(K)\to C^{**}$
such that $PJ_{\vert C}=\kappa_C$.

Since $M$ is a dual algebra, $\pi$ admits a (necessarily unique)
$w^*$-continuous extension $\widetilde{\pi}\colon C^{**}\to M$,
which is a $*$-homomorphism (see e.g. \cite[2.5.5]{BLM}). Then
$$
\widehat{\pi} = \widetilde{\pi}\circ P_{\vert A}
$$
is  a completely positive extension of $\pi$.
\end{proof}

We note that the above proof cannot be applied beyond the WEP case.

We now turn to an observation which will imply that the extension
property assumption in Theorem \ref{main} is actually necessary 
for the conclusion to hold, up to changing $M$ into a smaller 
von Neumann algebra containing the range of $u$.

In the rest of this section we consider $A,C,M,\pi$ as in Theorem \ref{main} and we
let $\pi(C)'\subset M$ be the commutant of the range of $\pi$.
If $e\in \pi(C)'$ is a projection, we let $\pi_e\colon C\to eMe$
be defined by $\pi_e(c)=e\pi(c)e$; this is a unital $*$-homomorphism.

\begin{proposition}\label{e1}
Let $\phi\colon A\to M$ be a $C$-bimodule completely positive 
map. Then there exists a projection $e\in\pi(C)'$ such that 
$\phi(A)\subset eMe$ and $\pi_e\colon C\to eMe$ admits a 
completely positive extension $A\to eMe$.
\end{proposition}

\begin{proof}
Let $b=\phi(1)$, then $b$ belongs to $M_+$. Let $e$ be 
the support projection
of $b$. Since $\phi$ is a $C$-bimodule map, we have 
$$
\pi(c) \phi(1) = \phi(c) = \phi(1)\pi(c)
$$
for any $c\in C$. Hence $b\in \pi(C)'$, which implies
that $b^{\frac{1}{2}}\in \pi(C)'$ and $e\in \pi(C)'$. Thus we have
\begin{equation}
\label{CE1}
\phi(c)= b^{\frac{1}{2}} \pi(c)b^{\frac{1}{2}},\qquad c\in C.
\end{equation}

For any integer $n\geq 1$ let $\psi_n\colon A\to M$ be defined by
$$
\psi_n(x) = \bigl(b+\tfrac{1}{n}\bigr)^{-\frac{1}{2}}\phi(x)
\bigl(b+\tfrac{1}{n}\bigr)^{-\frac{1}{2}},\qquad x\in A.
$$ 
According to the proof of \cite[Lem. 2.2]{CE}, 
$\bigl(b+\tfrac{1}{n}\bigr)^{-\frac{1}{2}}b^{\frac{1}{2}}
\to e$ strongly and $\bigl(\psi_n(x)\bigr)_{n\geq 1}$ has a strong limit
for any $x\in A$. Let $\psi\colon A\to M$ be the resulting point-strong
limit of $\psi_n$. Then $\psi$ is completely positive and valued in
$eMe$. By (\ref{CE1}), we have
$$
\psi_n(c) = \bigl(b+\tfrac{1}{n}\bigr)^{-\frac{1}{2}}b^{\frac{1}{2}} 
\pi(c)b^{\frac{1}{2}}
\bigl(b+\tfrac{1}{n}\bigr)^{-\frac{1}{2}}
$$
for any $c\in C$. Hence
$$
\psi(c) =e\pi(c)e,\qquad c\in C,
$$
which means that $\psi$ extends $\pi_e$.

Let $x\in A$ with $0\leq x\leq 1$. By positivity, $0\leq \phi(x)\leq b$,
which implies that $\phi(x)\in eMe$. We deduce that $\phi(A)\subset eMe$.
\end{proof}

\begin{corollary}\label{e2}
Let $u\colon A\to M$ be a linear combination of $C$-bimodules 
completely positive maps.
Then there exists a projection $e\in\pi(C)'$ such that 
$u(A)\subset eMe$ and $\pi_e\colon C\to eMe$ admits a 
completely positive extension $A\to eMe$.
\end{corollary}

\begin{proof}
By assumption, we may write $u=(u_1-u_2) + i(u_3 -u_4)$, where
for any $j=1,\ldots,4$, $u_j$ is a $C$-bimodule completely positive map.
Then the sum 
$$
\phi= u_1 + u_2 + u_3 + u_4
$$
is a $C$-bimodule completely positive map.
Applying Proposition \ref{e1}, we obtain a projection $e\in\pi(C)'$ such that 
$\pi_e\colon C\to eMe$ admits a 
unital completely positive extension to $A$ and $\phi(A)\subset eMe$.

Let $x\in A$ with $0\leq x\leq 1$. For any $j$, we have $0\leq u_j(x)\leq \phi(x)$, 
hence $u_j(x)\in eMe$. Thus $u_j(A)\subset eMe$, and hence $u(A)\subset eMe$.
\end{proof}

Combining Theorem \ref{main} and Corollary \ref{e2}, we obtain the following characterization.

\begin{corollary}\label{p_converse}
Let $u\colon A\to M$ be a $C$-bimodule map and assume 
that $u$ is decomposable. Then the following assertions are equivalent.
\begin{itemize}
\item [(i)] There exist $C$-bimodule completely positive maps 
$u_1,u_2,u_3,u_4\colon A\to M$ such that 
$u=(u_1-u_2) + i(u_3 -u_4)$.

\item [(ii)] There exists a projection $e\in \pi(C)'$ such that 
$\pi_e \colon A \to eMe$ admits a completely positive extension
$\widehat{\pi}_e\colon A\to eMe$ and $u(A)\subset eMe$. 

\end{itemize}
\end{corollary}

\begin{remark}
It is easy to modify the proof of Theorem \ref{main} to obtain the following extension 
property: consider $A,C,M,\pi$ as in Theorem \ref{main}, let $F\subset A$ be a linear subspace
which is an operator $C$-bimodule in the sense 
that $c_1 fc_2\in F$ for any $f\in F$ and any $c_1,c_2\in C$. Let $u\colon F\to M$
be a $C$-bimodule bounded map. Assume that $u$ admits a decomposable extension
$\widehat{u}\colon A\to M$ and $\pi$ admits a completely positive extension
$\widehat{\pi}\colon A\to M$. Then $u$ admits an extension $v\colon A\to M$
which is a linear  combination of $C$-bimodule completely positive maps
$A\to M$.

Theorem \ref{main} corresponds to the case $F=A$.
In the case when $M$ is injective, the above  result correponds to 
Wittstock's extension Theorem \cite{W} (see also \cite[Thm. 3.6.2]{BLM}).
\end{remark}

\section{Additional remarks}

In this final section we provide supplementary observations on the 
decomposition problem considered in this paper.
We first note a uniqueness property about the $*$-homomorphism $\pi$ with respect to which a 
completely positive map $\phi\colon A\to M$ can be considered as a $C$-bimodule map. In the sequel
we write $e^\perp=1-e$ for a projection $e\in M$.

\begin{corollary}\label{ce3} Consider a von Neumann algebra $M$, a $C^*$-algebra $A$
and a unital sub-$C^*$-algebra $C\subset A$.
Let $\pi_1\colon A\to M$ and $\pi_2\colon A\to M$ be unital $*$-homomorphisms and let 
$\phi\colon A\to M$ be a completely positive map. Assume that $\phi$ is
a $C$-bimodule map with respect to $\pi_1$. Then the following assertions are equivalent.
\begin{itemize}
\item  [(i)] $\phi$ is
a $C$-bimodule map with respect to $\pi_2$; 
\item  [(ii)] There exists a projection $e\in \pi_1(C)^\prime\cap \pi_2(C)^\prime$ 
such that $\phi(A)\subset eMe$ and $(\pi_1-\pi_2)(C)\subset e^\perp Me^\perp$.
\end{itemize}
\end{corollary}

\begin{proof}  
Suppose firstly that $\phi\colon A\to M$ is a $C$-bimodule map with respect to $\pi_2$. 
Then as in the proof of Proposition \ref{e1},  the support 
projection $e$ of $\phi(1)$ lies in $\pi_1(C)^\prime\cap \pi_2(C)^\prime$ 
and the map $\psi\colon A\to eMe$
is a completely positive extension of the maps 
$({\pi_1})_e$ and  $({\pi_2})_e$. 
Consequently, for all $c\in C$, 
\begin{align*}
\pi_1(c) - \pi_2(c) &= e\pi_1(c)e +e^\perp\pi_1(c)e^\perp - 
e\pi_2(c)e -e^\perp\pi_2(c)e^\perp \\
&= e^\perp(\pi_1(c)-\pi_2(c))e^\perp,
\end{align*}
and, therefore, 
$(\pi_1-\pi_2)(C)\subset e^\perp Me^\perp$.

Conversely, let    
$e\in  \pi_1(C)^\prime\cap \pi_2(C)^\prime$ 
be a projection such that $\phi(A)\subset eMe$ and $(\pi_1-\pi_2)(C)\subset e^\perp Me^\perp$.
Then given any $c\in C$ there exists $z\in  e^\perp Me^\perp$ such that
$\pi_1(c)=\pi_2(c) +z$. Hence for all $a\in A$, 
\begin{align*}
\phi(ca)=\pi_1(c)\phi(a)&=(\pi_2(c)+z)\phi(a)\\
&=\pi_2(c)\phi(a)+z\phi(a)\\
&=\pi_2(c)\phi(a)+e^\perp z e^\perp \phi(a).
\end{align*}
Since $\phi$ is valued in $eMe$, this implies
$$
\phi(ca) =\pi_2(c)\phi(a).
$$
Similarly we have
$$
\phi(ac)=\phi(a)\pi_2(c).
$$
\end{proof}

Our second observation is about the non unital case. Let $A$ be a non unital 
$C^*$-algebra. Let $M(A)$ denote its multiplier algebra, that we may regard as a unital
sub-$C^*$-algebra of $A^{**}$ (see e.g. \cite[2.6.7]{BLM}). Let $C\subset M(A)$
be a unital $C^*$-algebra, let $M$ be a von Neumann algebra and let 
$\pi\colon C\to M$ be a unital $*$-homomorphism. The definition
of a $C$-bimodule map $u\colon A\to M$ 
with respect to $\pi$ as given by (\ref{Bim}) extends to that setting. Then 
Theorem \ref{main} extends as follows.

\begin{corollary}
Let $u\colon A\to M$ be a $C$-bimodule map (with respect to $\pi$) and assume that 
$u$ is decomposable. If $\pi$ admits a completely positive extension
$A^{**}\to M$ then there exist $C$-bimodule completely positive maps
$u_1,u_2,u_3,u_4\colon A\to M$ such that 
$u=(u_1-u_2) + i(u_3 -u_4)$.
\end{corollary}

\begin{proof} Let $\widetilde{u}\colon A^{**}\to M$ be the unique 
$w^*$-continuous extension of $u$ (see e.g. \cite[Lem. A.2.2]{BLM}.
It is easy to check that $\widetilde{u}$ is decomposable, with 
$\decnorm{\widetilde{u}}=\decnorm{u}$. Since the product is separately 
$w^*$-continuous on von Neumann algebras, $\widetilde{u}$ is 
a $C$-bimodule map with respect to $\pi$. Indeed let $c_1,c_2\in C$,
let $\eta\in A^{**}$ and let $(a_i)_i$ be a net in $A$ such that
$a_i\to\eta$ in the $w^*$-topology of $A^{**}$. Then 
$c_1a_ic_2\to c_1\eta c_2$ in the $w^*$-topology of $A^{**}$, hence
$$
\widetilde{u}(c_1\eta c_2) =w^*-\lim_i u(c_1a_ic_2).
$$
Further $u(c_1a_ic_2)=\pi(c_1)u(a_i)\pi(c_2)$ and $u(a_i)
\to\widetilde{u}(\eta)$ in the $w^*$-topology of $M$. Hence
$$
w^*-\lim_i \pi(c_1)u(a_i)\pi(c_2) = \pi(c_1)\widetilde{u}(\eta)\pi(c_2).
$$
This yields $\widetilde{u}(c_1\eta c_2)=
\pi(c_1)\widetilde{u}(\eta)\pi(c_2)$.

Applying Theorem \ref{main} to $\widetilde{u}$ we obtain a decomposition 
of that map into completely positive $C$-bimodule maps 
$A^{**}\to M$. Restricting to $A$, we find the desired decomposition of $u$.
\end{proof}

Our final observation is that the decomposability problem considered in this paper 
always has a positive solution if $\pi$ is a $*$-isomorphism. This is a direct
consequence of a remarkable projection result of Christensen-Sinclair \cite{CS}. 
We thank \'Eric Ricard for pointing out this
result to us.

\begin{proposition} Let $A,C,M$ as in Theorem \ref{main} and let $\pi\colon C\to M$ is a $*$-isomorphism. Then 
for any decomposable and $C$-bimodule map $u\colon A\to M$, there exist 
$C$-bimodule completely positive maps $u_1,u_2,u_3,u_4\colon A\to M$ such that
$u=(u_1-u_2) +i(u_3-u_4)$.
\end{proposition}

\begin{proof} Let $CB(A,M)$ be the space of completely bounded maps from 
$A$ into $M$, equipped with $\cbnorm{\ }$. 
Let $BIMOD(A,M)\subset CB(A,M)$ be the subspace of $C$-bimodule
completely bounded maps. According to \cite[Thm. 4.1]{CS} (to be applied with
$\theta=\pi^{-1}$), there exists a contractive idempotent map 
$Q\colon CB(A,M)\to CB(A,M)$ with range equal to
$BIMOD(A,M)$, such that $Q(v)$ is completely positive for any 
completely positive $v\colon A\to M$.

Let $u\colon A\to M$ be decomposable and write it as $u=(v_1-v_2) +i(v_3-v_4)$ 
for some completely positive maps $v_j\colon A\to M$.
Then $Q(u)=(Q(v_1)-Q(v_2)) +i(Q(v_3)-Q(v_4))$ and each $Q(v_j)$ is a completely
positive and $C$-bimodule map. If $u$ is assumed to be $C$-bimodule,
then $Q(u)=u$ and the above decomposition proves the result.
\end{proof}

\end{document}